\documentclass[12pt,leqno]{amsart}
\usepackage{latexsym,amsmath,amssymb}
\usepackage{graphicx,hyperref}

\title[Weak BLD mappings]{Weak BLD mappings and Hausdorff measure}

\author{Piotr Haj{\l}asz,  Soheil Malekzadeh and Scott Zimmerman}

\address{P.\ Haj{\l}asz: Department of Mathematics, University of Pittsburgh, 301
  Thackeray Hall, Pittsburgh, PA 15260, USA, {\tt hajlasz@pitt.edu}}

\address{S. Malekzadeh: Department of Mathematics, University of Pittsburgh, 301
  Thackeray Hall, Pittsburgh, PA 15260, USA, {\tt som13@pitt.edu}}

\address{S. Zimmerman: Department of Mathematics, University of Pittsburgh, 301
  Thackeray Hall, Pittsburgh, PA 15260, USA, {\tt srz5@pitt.edu}}

\thanks{P.H.\ was supported by NSF grant DMS-1500647.}

plus 6pt minus 12pt
\def\rank{{\rm rank\,}}
\newcommand{\cJ}{\mathcal J}

\def\eps{\varepsilon}

\def\id{{\rm id\, }}

\def\H{{\mathcal H}}

\newtheorem{theorem}{Theorem}
\newtheorem{lemma}[theorem]{Lemma}

\def\diam{{\rm diam\,}}

\theoremstyle{definition}

\newcommand{\barint}{
\rule[.036in]{.12in}{.009in}\kern-.16in \displaystyle\int }

\newcommand{\barcal}{\mbox{$ \rule[.036in]{.11in}{.007in}\kern-.128in\int $}}

\newcommand{\bbbn}{\mathbb N}

\newcommand{\bbbr}{\mathbb R}

\newcommand{\Heis}{\mathbb H}

\def\ap{\operatorname{ap}}

\def\diam{\operatorname{diam}}

\def\mvint_#1{\mathchoice
          {\mathop{\vrule width 6pt height 3 pt depth -2.5pt
                  \kern -8pt \intop}\nolimits_{\kern -3pt #1}}%
          {\mathop{\vrule width 5pt height 3 pt depth -2.6pt
                  \kern -6pt \intop}\nolimits_{#1}}%
          {\mathop{\vrule width 5pt height 3 pt depth -2.6pt
                  \kern -6pt \intop}\nolimits_{#1}}%
          {\mathop{\vrule width 5pt height 3 pt depth -2.6pt
                  \kern -6pt \intop}\nolimits_{#1}}}

\numberwithin{theorem}{section} \numberwithin{equation}{section}

\begin{document}

\sloppy


\begin{abstract}
We prove that if $\Phi:X\to Y$ a mapping of weak bounded length distortion from
a quasiconvex and complete metric space $X$ to any metric space $Y$, then 
for any Lipschitz mapping $f:\bbbr^k\supset E\to X$ we have that
$\H^k(f(E))=0$ in $X$ if and only if $\H^k(\Phi(f(E)))=0$ in $Y$.
This generalizes an earlier result of Haj\l{}asz and Malekzadeh where
the target space $Y$ was a Euclidean space $Y=\bbbr^N$.
\end{abstract}

\maketitle

\centerline{\em To Carlo Sbordone on his 70th birthday}

\section{Introduction}

A mapping $f : X \to Y$ between metric spaces is said to have a 
{\em weak bounded length distortion} ({\em weak} BLD) property 
if there is a constant $M \geq 1$ such that, for all {\em rectifiable} curves $\gamma$ in $X$, 
the length of $f\circ\gamma$ is comparable to that of $\gamma$ 
in the following sense:
\begin{equation}
\label{eq1}
M^{-1}\ell_X(\gamma) \leq \ell_Y(f \circ \gamma) \leq M \ell_X(\gamma).
\end{equation}
This definition was introduced in \cite{HM1,HM2} and it was motivated by earlier work of
Martio and V\"ais\"al\"a \cite{MV} and Le Donne \cite{ledonne}.

Martio and V\"ais\"al\"a \cite{MV} introduced mappings of {\em bounded length distortion} (BLD).
These are mappings $f:\bbbr^n\supset\Omega\to\bbbr^n$ defined on an open set $\Omega$ 
that are open, discrete, sense preserving and satisfy
\eqref{eq1} for {\em all} curves $\gamma$ in $\Omega$, see also \cite{HM2}. Subsequently, 
Le Donne \cite{ledonne} introduced mappings of {\em bounded length distortion} (BLD) as mappings between metric spaces
that satisfy \eqref{eq1} for {\em all} curves $\gamma$ in $X$, but without the topological requirements of
being open, discrete, or sense preserving. 

The class of BLD mappings plays a fundamental role in the 
contemporary development of geometric analysis and geometric topology, especially in the context of branched coverings of metric spaces. 
See e.g. 
\cite{drasinp,heinonenk,HK,HKM,heinonenr,heinonenr2,heinonens,ledonne,ledonnep,luisto,pankka}.

It is important to observe that, in general, the class of weak BLD mappings may be much
larger than the class of BLD mappings given by Le Donne. Indeed, the identity mapping
$\id:\Heis^n\to\bbbr^{2n+1}$ from the Heisenberg group $\Heis^n$ into Euclidean space is weak BLD. 
However, it is not BLD
since it maps the $t$-axis, which has Hausdorff dimension two with respect to the Carnot-Carath\'eodory metric, 
to the Euclidean $t$-axis
which has locally finite length.

The aim of this paper is to prove Theorem~\ref{main}. 
This generalizes Theorem~4.2 from \cite{HM1} in which the same statement was proven for
weak BLD mappings from $X$ into a Euclidean space $\bbbr^N$. 
Our proof will follow a similar argument as in \cite{HM1}. 
However, a new proof is required as
the co-domain $Y$ is no longer Euclidean but is instead an arbitrary metric space. 
The main difference between the proofs appears at the end 
where we apply Lemma~\ref{separable} and estimate the length of the curve 
$\Phi\circ\Gamma$.
The arguments in the proof which are in 
\cite{HM1} will only be sketched,
and we refer the reader to \cite{HM1} for more details.

A metric space $(X,d)$ is said to be \emph{quasiconvex} 
if there is a constant $C_q \geq 1$ such that, 
for any $x,y\in X$, there is a rectifiable curve $\gamma:[0,1]\to X$ connecting
$x$ and $y$ (i.e. $\gamma(0)=x$ and $\gamma(1)=y$) 
whose length satisfies $\ell(\gamma)\leq C_q d(x,y)$. 
Such a curve $\gamma$ will be called \emph{quasiconvex}.
Note that, if $(X,d)$ is quasiconvex, 
then any weak BLD mapping $f:X\to Y$ is $MC_q$-Lipschitz.

\begin{theorem}
\label{main}
Let $(X,d_X)$ be a complete and quasiconvex metric space, 
and let $(Y,d_Y)$ be any metric space. 
Let $\Phi:X\to Y$ be a weak BLD mapping. 
Then for any $k\in \bbbn$ and any Lipschitz map $f:E \to X$ 
defined on a measurable set $E\subset \mathbb{R}^k$, the following conditions are equivalent:
\begin{enumerate}
\item $\H^k(f(E)) = 0$ in $X$,
\item $\H^k(\Phi(f(E))) = 0$ in $Y$.
\end{enumerate}
\end{theorem}

If there are no rectifiable curves in $X$, then any mapping $\Phi:X\to Y$ is weak BLD.
Thus the assumption that
the space $X$ is quasiconvex is a very natural one.
Clearly, bi-Lipschitz mappings preserve sets of Hausdorff measure zero, 
but the weak BLD condition is much weaker than bi-Lipschitz continuity. 
Recall, the identity map from the
Heisenberg group $\Heis^n$ to $\bbbr^{2n+1}$ is weak BLD.
This together with Theorem~\ref{main} can be used to prove unrectifiabilty of the Heisenberg group (see \cite{HM1}).

Another application of the theorem is to a result of Gromov. 
In \cite[Theorem 2.4.11]{gromov-five}, Gromov proved that any Riemannian manifold of dimension $n$ admits a mapping into 
$\mathbb{R}^n$ that preserves lengths of curves. 
It follows from Theorem~\ref{main} that the Jacobian of such a mapping is different than zero almost everywhere, 
and hence there is no curve-length preserving mapping into $\mathbb{R}^m$ for $m < n$. 
While this result is known, Theorem~\ref{main} provides a new perspective. 
For other comments and applications see \cite{HM1} and \cite{HM2}.

We will prove Theorem~\ref{main} as a consequence of the following result.
\begin{theorem}
\label{main-quasi}
Suppose that $(X,d)$ is a complete and quasiconvex metric space, 
and let $\Phi : X \to \ell^\infty$ be a weak BLD mapping.
Then, for any $k \in \bbbn$ and any Lipschitz map $f : E \to X$ defined on a measurable set 
$E \subset \mathbb{R}^k$, the following conditions are equivalent:
\begin{enumerate}
\item $\H^k(f(E)) = 0$ in $X$,
\item $\H^k(\Phi(f(E))) = 0$ in $\ell^{\infty}$,
\item $\rank(\ap D(\Phi\circ f)) < k$ a.e. in $E$.
\end{enumerate}
\end{theorem}

The last condition (3) requires some explanation.
Let $g = (g_1, g_2, \ldots) : \mathbb{R}^k \supset E \to \ell^\infty$ be an $L$-Lipschitz mapping.
Then the components $g_i : E \to \mathbb{R}$ are also $L$-Lipschitz. 
Hence, for $\H^k$-almost all points $x \in E$, 
the functions $g_i$, 
$i \in \mathbb{N}$ are approximately differentiable at $x \in E$. 
We define the approximate derivative of $g$ component-wise as follows:
$$
\ap Dg(x)=
\left\lceil
\begin{array}{c}
\ap Dg_1(x)\\
\ap Dg_2(x)\\
\vdots
\end{array}
\right\rceil
$$
For each $i \in \mathbb{N}$, $\ap Dg_i(x)$ is a vector in $\mathbb{R}^k$ with components bounded by $L$.
Thus $\ap Dg(x)$ can be regarded as an $\infty \times k$ matrix of real numbers whose components are bounded by $L$.
It is easy to see that, for an $\infty \times k$ matrix, the row rank equals the column rank.
Indeed, the rank-nullity theorem still holds for such matrices.
Therefore, the rank of $\ap Dg(x)$ is always at most $k$.

The paper is structured as follows. In Section~\ref{2} we show how to deduce Theorem~\ref{main} from Theorem~\ref{main-quasi}, 
and in Section~\ref{3} we prove Theorem~\ref{main-quasi}.

\section{Proof of Theorem~\ref{main} from Theorem~\ref{main-quasi}}
\label{2}

If $Y$ is a separable space, Theorem~\ref{main} follows very easilly from Theorem~\ref{main-quasi}.
Indeed, every separable metric space admits an isometric (Kuratowski) embedding $\kappa:Y\to\ell^\infty$,
and the composition $\kappa\circ\Phi:X\to\ell^\infty$
is still a weak BLD mapping. 
Thus for any Lipschitz mapping $f:\bbbr^k \supset E \to X$, Theorem~\ref{main-quasi} implies that
$\H^k(f(E)) = 0$ in $X$ 
if and only if 
$\H^k((\kappa\circ\Phi)(f(E))) = 0$ in $\ell^{\infty}$.
However, the last condition is equivalent to $\H^k(\Phi(f(E)))=0$ since isometries preserve Hausdorff measure and hence
$$
\H^k(\Phi(f(E))) = \H^k(\kappa(\Phi(f(E)))).
$$

If $Y$ is not separable, the arguments are slightly more complicated.
The metric space $\Phi(f(E))\subset Y$ with the induced metric is separable, 
so it admits an isometric embedding 
$\kappa:\Phi(f(E))\to\ell^\infty$.
We would like to mimic the above argument, but there is a technical issue:
the mapping $\kappa\circ\Phi$
is defined only on the set $f(E)\subset X$, 
and in general this set is neither quasiconvex nor complete as a metric space with the induced metric. 
It is, however, a separable subset of $X$.
We may thus use the following lemma to reduce $X$ to a separable, quasiconvex, complete space $\tilde{X}$ containing $f(E)$.

\begin{lemma}
\label{separable}
Let $(X,d)$ be a quasiconvex and complete metric space and let $A\subset X$ be a separable subset. Then there is a subset 
$\tilde{X}\subset X$ containing $A$ such that $(\tilde{X},d)$ is separable, quasiconvex, and complete.
\end{lemma}

Before proving the lemma, we will show how to use it to complete the proof of Theorem~\ref{main}.
Set $A = f(E)$ and choose the space $\tilde{X}$ as in the lemma.
Since $\tilde{X}$ is separable, so too is $\Phi(\tilde{X}) \subset Y$.
We thus have an isometric embedding 
$\tilde{\kappa}:\Phi(\tilde{X})\to\ell^\infty$,
and so the mapping $\tilde{\kappa} \circ \Phi: \tilde{X} \to \ell^{\infty}$ is weak BLD.
Thus it follows from Theorem~\ref{main-quasi}
that $\H^k(f(E)) = 0$ in $\tilde{X}$ (and thus in $X$) 
if and only if 
$$
\H^k(\Phi(f(E))) = \H^k(\tilde{\kappa}(\Phi(f(E)))) = 0
$$
in $Y$.
This completes the proof of Theorem~\ref{main}.
It remains to prove the lemma.

\begin{proof}[Proof of Lemma~\ref{separable}]
Choose a countable and dense subset $A_1=\{x_i\}_{i=1}^{\infty}$ of $A$.
For any $i,j \in \bbbn$, 
choose a quasiconvex curve $\gamma_{ij}:[0,1] \to X$ connecting $x_i$ to $x_j$.
Define the set 
$
\cJ_1 := \{ \gamma_{ij} \}_{i,j=1}^{\infty}.
$ 
This is a countable family of quasiconvex curves connecting all pairs of points in $A_1$.

Suppose by way of induction that the countable set $A_n$ and countable family of curves $\cJ_n$ have been defined.
Define $A_{n+1}$ to be a countable, dense subset of $\bigcup_{\gamma \in \cJ_n} \gamma([0,1])$ such that $A_{n+1} \supset A_n$.
This is possible since each $\gamma([0,1])$ is separable and since the family $\cJ_n$ is countable.
As above, define $\cJ_{n+1}$ to be a countable family of quasiconvex curves connecting all pairs of points in $A_{n+1}$
by selecting one curve for each pair of points.

Set $\tilde{A}=\bigcup_{n=1}^\infty A_n$,
and define $\tilde{X}$ to be the closure of $\tilde{A}$. 
Clearly, $A\subset\tilde{X}$, 
and $\tilde{X}$ is separable. 
Moreover, $\tilde{X}$ is complete as it is a closed subset of a complete space.
It remains to show that $(\tilde{X},d)$ is quasiconvex.
If $x,y\in\tilde{A}$, then $x,y\in A_n$ for some $n$ (because $A_1\subset A_2\subset A_3\subset\ldots$). 
Hence the points $x$ and $y$ can be connected by a quasiconvex curve $\gamma$ that belongs to $\cJ_n$. 
Since $A_{n+1}\cap\gamma([0,1])$ is dense in $\gamma([0,1])$, we have that $\gamma([0,1])\subset\tilde{X}$. 

Let $\eps>0$.
Fix $x,y\in \tilde{X}$ with $x\neq y$.
Then there are sequences $\{ x_k\}_{k=1}^\infty$ and $\{ y_k\}_{k=1}^\infty$ in $\tilde{A}$ with
$$
d(x_{k},x)\leq \eps d(x,y) 2^{-k}
\quad
\text{and}
\quad
d(y_{k},y)\leq \eps d(x,y) 2^{-k}.
$$
It easily follows from the triangle inequality that
$$
d(x_k,x_{k+1})\leq \eps d(x,y) 2^{-(k-1)}
\quad
\text{and}
\quad
d(y_k,y_{k+1})\leq \eps d(x,y) 2^{-(k-1)}
$$
Also 
$$
d(x_1,y_1)\leq (1+\eps)d(x,y).
$$
Hence
$$
d(x_1,y_1) + \sum_{k=1}^\infty d(x_k,x_{k+1}) + \sum_{k=1}^\infty d(y_k,y_{k+1}) \leq (1+5\eps)d(x,y).
$$
By the arguments above, 
we may connect $x_1$ and $y_1$ by a quasiconvex curve $\gamma_0$ in $\tilde{X}$ of length at most $C_q \, d(x_1,y_1)$.
Here, $C_q$ is the quasiconvexity constant associated with $X$.
For $k\in\bbbn$ we connect $x_k$ to $x_{k+1}$ by a quasiconvex curve 
$\alpha_k$ in $\tilde{X}$ of length at most $C_q \, d(x_k,x_{k+1})$
and connect $y_k$ to $y_{k+1}$ by a quasiconvex curve 
$\beta_k$ in $\tilde{X}$ of length at most $C_q \, d(y_k,y_{k+1})$.
Concatenating these curves in the obvious order creates a rectifiable curve $\gamma$
in $\tilde{X}$ with length
$$
\ell(\gamma) = \ell(\gamma_0) + \sum_{k=1}^{\infty} \ell(\alpha_k)  + \sum_{k=1}^{\infty} \ell(\beta_k)
\leq C_q (1+5\eps) d(x,y).
$$
The curve $\gamma$ connects $x$ and $y$. 
Therefore, the space $\tilde{X}$ is quasiconvex with any quasiconvexity constant larger than $C_q$.
\end{proof}

\section{Proof of Theorem~\ref{main-quasi}}
\label{3}

Suppose that $(X,d)$ is complete and quasiconvex 
and $\Phi : X \to \ell^\infty$ is weak BLD.
Suppose also that $E \subset \bbbr^k$ and $f : E \to X$ is Lipschitz.
The implication from (1) to (2) is obvious because the mapping $\Phi$ is Lipschitz. 
The equivalence between (2) and (3) follows from the following
result \cite[Theorem 2.2]{HM1}:

\begin{lemma}
\label{intulinfty}
Let $E \subset \mathbb{R}^k$ be measurable and let $g : E \to \ell^\infty$ be a Lipschitz mapping. Then $\H^k(g(E))=0$ if and only if $\rank(\ap Dg(x)) < k$, 
$\H^k$-a.e. in $E$.
\end{lemma}

The implication from left to right is easy: $g$
composed with a projection $\pi$ of $\ell^{\infty}$ onto any $k$-dimensional subspace 
generated by a choice of $k$-coordinates in $\ell^\infty$
is Lipschitz, so $\H^k(\pi(g(E)))=0$. 
Hence, by the area formula, the 
determinant of the mapping $\pi\circ g$ equals zero a.e. 
This implies that $\rank(\ap Dg(x)) < k$ a.e.

The reverse implication is much more difficult
and follows the Sard type arguments seen in the remainder of this paper.
For details, see \cite{HM1}.

It remains to prove that (3) implies (1). 
Suppose that $\rank (\ap D(\Phi\circ f))<k$ a.e. in $E$. 
Let $\eps>0$. 
Since $(\Phi \circ f)_i:E \to \bbbr$ is Lipschitz for each $i \in \bbbn$,
we can find $g_i \in C^1(\bbbr^k)$ with
$$
\H^k(\{x \in E \, : \, (\Phi\circ f)_i(x) \neq g_i(x)\}) < \eps/2^i
$$
and $\ap D((\Phi \circ f)_i)(x) = Dg_i(x)$ for almost every $x \in E$ at which $(\Phi \circ f)_i(x)=g_i(x)$.
Thus there is a measurable set $F\subset E$ such that $\H^k(E\setminus F)<\eps$ and
$$
\Phi\circ f = (g_1,g_2,\ldots)=:g,
\qquad
\ap D(\Phi\circ f)=
\left\lceil
\begin{array}{c}
Dg_1\\
Dg_2\\
\vdots
\end{array}
\right\rceil\, =: Dg
$$
at all points of $F$. Let
$$
\tilde{F}=\{x\in F:\, \rank \ap D(\Phi\circ f)(x)=\rank Dg(x)<k\}.
$$
By assumption, $\H^k(F\setminus\tilde{F})=0$.
Recall that our goal is to prove $\H^k(f(E))=0$.
It suffices to prove $\H^k(f(\tilde{F}))=0$ 
since we may exhaust $E$ by sets $\tilde{F}$ up to a set of $\H^k$-measure zero
and since $f$ maps sets of $\H^k$-measure zero to sets of 
$\H^k$-measure zero.
Moreover, the set $\tilde{F}$
can be decomposed as follows:
$$
\tilde{F}=\bigcup_{j=0}^{k-1} K_j,
\qquad
K_j=\{x\in\tilde{F}:\, \rank Dg(x)=j\}.
$$
Thus it suffices to show that $\H^k(f(K_j))=0$ for $j=0,1,2,\ldots,k-1$.
By removing a set of measure zero we can assume that all points of $K_j$ are density points of $K_j$.

In fact, it suffices to prove that any point in $K_j$ has a cubic neighborhood 
whose intersection with $K_j$ is mapped onto a set of $\H^k$-measure zero. 
In the remainder of the paper, a ``cube'' will refer to a cube with edges parallel to the coordinate axes.

For each $j \geq 1$, we may 
apply the change of variables \cite[Lemma 2.6]{HM1} 
as in the proof of \cite[Theorem 2.2]{HM1}
and assume that
\begin{equation}
\label{eq12}
K_j\subset (0,1)^k 
\quad
\text{and}
\quad
\text{$g_i(x)=x_i$ for $i=1,2,\ldots,j$ and $x\in [0,1]^k$.}
\end{equation}
Since $\rank Dg(x)=j$ for any $x\in K_j$ 
and since $g$ fixes the first $j$ coordinates of $x$, we have
\begin{equation}
\label{eq122}
\frac{\partial g_{\ell}}{\partial x_i}(x)=0
\quad
\text{for $x\in K_j$, $i=j+1,\ldots,k$ and all $\ell=1,2,\ldots$}
\end{equation}
If $j=0$ we do not need to apply a change of variables.

Now the result will follow from the next lemma after a standard application of the Vitali type $5r$-covering lemma. 
Indeed, it allows us to cover $K_j$ by cubes $Q_{x_i}$ 
so that the cubes $5^{-1}Q_{x_i}$ are pairwise disjoint
and thus bound the Hausdorff content $\H^k_\infty(f(K_j))$ by $C m^{j-k}$ for some $C>0$,
and this can be made arbitrarily small
since $m$ is arbitrary and $j-k < 0$.
See the argument following the statement of Lemma 2.7 in \cite{HM1} for full details.
\begin{lemma}
Let $M$ be the BLD constant of $\Phi$, 
let $C_q$ be the quasiconvexity constant of $X$,
and let $L$ be the Lipschitz constant of $f$.
Under the assumptions \eqref{eq12}, 
there is a constant $C = C(k)C_q^2M^2 > 0$ 
such that, for any integer $m \geq 1$ and $x \in K_j$, 
there is a closed cube $Q_x \subseteq [0, 1]^k$ centered at $x$ of edge length $d_x$ such that 
$f (K_j \cap Q_x)$ can be covered by $m^j$ balls in $X$, 
each of radius $CLd_xm^{-1}$.
\end{lemma}
\begin{proof}
Since $x$ is a density point of $K_j$, 
there is a cube $Q \subseteq [0, 1]^k$ centered at $x$ with edge length $d$
such that $\H^k(Q \setminus K_j) < m^{-k}d^k$. 
We can assume that $Q = [0, d]^k$. 
Divide $[0,d]^j$ into cubes $\{ Q_\nu \}_{\nu = 1}^{m^j}$ of edge length $d/m$
with pairwise disjoint interiors. 
We want to prove that, inside each rectangular box $Q_\nu \times [0, d]^{k-j}$, the set $K_j$ is mapped by $f$ into a small ball.
In particular, we want 
\begin{equation}
\label{boxshrink}
\diam_X ( f((Q_\nu \times [0, d]^{k-j}) \cap K_j) ) < C(k)C_q^2M^2Ldm^{-1}.
\end{equation}

Since $\H^k((Q_\nu\times [0, d]^{k-j})\setminus K_j)\leq \H^k(Q \setminus K_j) < m^{-k}d^k$,
we may use Fubini's theorem as in the proof of \cite[Lemma 2.7]{HM1} 
to find $\rho \in Q_\nu$ such that
\begin{equation}
\label{43}
\H^{k-j}((\{\rho\} \times [0, d]^{k-j}) \setminus K_j) < m^{j-k}d^{k-j}.
\end{equation}
In particular, 
every point in $\{\rho\} \times [0, d]^{k-j}$ is at a distance no more than $C(k)m^{-1}d$ from the set
$(\{\rho\} \times [0, d]^{k-j})\cap K_j$. 
Hence every point in $Q_\nu \times [0, d]^{k-j}$ 
(and thus every point in $(Q_\nu \times [0, d]^{k-j}) \cap K_j$) 
is at a distance at most $C(k)m^{-1}d$ from the set 
$(\{\rho\} \times [0, d]^{k-j}) \cap K_j$. 
Since $f$ is $L$-Lipschitz,
in order to prove \eqref{boxshrink}
it suffices to show that
\begin{equation}
\label{44}
\diam_X(f((\{\rho\}\times [0,d]^{k-j})\cap K_j)) < C(k)C_q^2M^2Ldm^{-1}.
\end{equation}

To begin to prove \eqref{44}, we will recall Lemma~4.4 from \cite{HM1}.
\begin{lemma}
\label{si}
Let $E \subset \mathcal{Q}$ be a measurable subset of a cube $\mathcal{Q} \subset \mathbb{R}^n$.
Then there is a constant $C=C(n)>0$ such that, for any $x\in \mathcal{Q}$,
\begin{equation}
\label{blabla}
\H^n(\{y\in \mathcal{Q}:\, \H^1(\overline{xy}\cap E) \leq C\H^n(E)^{1/n}\})>\frac{\H^n(\mathcal{Q})}{2}
\end{equation}
where $\overline{xy}$ is the segment from $x$ to $y$.
\end{lemma}
This lemma implies that, if the measure of $E \subset \mathcal{Q}$ is small,
then more than half of the intervals in $\mathcal{Q}$ intersect $E$ along a short subset.
See \cite[Lemma 4.4]{HM1} for a short proof.

Under the assumptions of Lemma~\ref{si}, 
for any $x, y \in \mathcal{Q}$ we can find $z \in \mathcal{Q}$ such that 
$$
\H^1(\overline{xz}\cap E)+\H^1(\overline{zy}\cap E) \leq C\H^n(E)^{1/n}.
$$
That is, the curve $\overline{xz}+\overline{zy}$ connecting $x$ to $y$ intersects the set $E$ along a subset 
of length at most $C\H^n(E)^{1/n}$.
Notice also that this curve has length no larger than $2 \diam(\mathcal{Q})$. 

Applying this argument with dimension $n = k-j$, 
cube $\mathcal{Q} = \{\rho\}\times[0, d]^{k-j}$, 
and subset $E = ( \{\rho\}\times[0, d]^{k-j})\setminus K_j$, every pair of 
points $x, y \in (\{\rho\}\times[0, d]^{k-j}) \cap K_j$ can be connected by a curve
of length at most $2d\sqrt{k-j}$ 
(which is two times the diameter of $\{\rho\}\times[0, d]^{k-j}$) 
whose intersection with $(\{\rho\}\times[0, d]^{k-j}) \setminus K_j$ has length no more than $C(k)m^{-1}d$ (by (\ref{43})).

Fix $x,y \in (\{\rho\}\times[0, d]^{k-j}) \cap K_j$
and choose $\gamma$ to be a curve in $\{\rho\}\times[0, d]^{k-j}$
as described in the previous paragraph.
Parametrize $\gamma$ by arc-length so that it is a $1$-Lipschitz curve.
The mapping $f\circ\gamma$ is $L$-Lipschitz 
and is defined on the set $\gamma^{-1}(K_j) \subset [0,\ell(\gamma)]$. 
This map uniquely extends to an $L$-Lipschiz map defined on the closure 
of $\gamma^{-1}(K_j)$ 
(since it is Lipschitz and $X$ is complete). 
The complement of $\overline{\gamma^{-1}(K_j)}$ consists of countably many (relatively) open intervals whose total 
length is bounded by $C(k)m^{-1}d$. 
Since the space $X$ is quasiconvex, 
we can extend $f\circ\gamma$ from $\overline{\gamma^{-1}(K_j)}$ 
to a $C_qL$-Lipschitz curve $\Gamma:[0,\ell(\gamma)]\to X$ connecting $f(x)$ to $f(y)$.
Indeed, we may construct this extension by choosing 
for each open interval in the complement of $\overline{\gamma^{-1}(K_j)}$
a quasiconvex curve in $X$
(which is $C_qL$-Lipschitz on the interval after possibly reparameterizing)
that connects the images of the endpoints of the interval.

According to the paragraph preceding the statement of Theorem~\ref{main}, 
the mapping $\Phi$ is $MC_q$-Lipschitz so
the curve $\Phi\circ\Gamma:[0,\ell(\gamma)] \to \ell^\infty$
is $MC_q^2L$-Lipschitz. 
In order to prove \eqref{44}, it suffices to show that 
\begin{equation}
\label{eq367}
\ell(\Phi\circ\Gamma)\leq C(k)MC_q^2Lm^{-1}d.
\end{equation}
Indeed, since $\Phi$ is weak BLD we would have
$$
d(f(x),f(y)) \leq \ell(\Gamma) \leq M \ell(\Phi\circ\Gamma) \leq C(k)C_q^2M^2Lm^{-1}d.
$$
Since we may find such a curve $\Gamma$ for any $x, y \in (\{\rho\} \times [0, d]^{k-j}) \cap K_j$, \eqref{44} follows.

Thus it remains to prove the estimate \eqref{eq367}. Since $\Phi \circ \Gamma$ is a curve in $\ell^\infty$, 
the proof of this estimate is slightly more
subtle than that of the corresponding estimate in the proof of \cite[Theorem~4.2]{HM1} where the curve was in $\bbbr^N$.
The proof is a result of the following lemma:

\begin{lemma}
If $\eta=(\eta_1,\eta_2,\ldots):[a,b]\to\ell^\infty$ is a Lipschitz curve, then
$$
\ell(\eta)\leq \int_a^b\Vert\eta'(t)\Vert_\infty\, dt,
\quad
\text{where $\eta'=(\eta_1',\eta_2',\ldots)$.}
$$
\end{lemma}
\begin{proof}
For $[s,t] \subset [a,b]$,
$$
\Vert \eta(t)-\eta(s)\Vert_\infty=\sup_i |\eta_i(t)-\eta_i(s)| \leq
\sup_i \int_s^t|\eta_i'(\tau)| \, d\tau \leq \int_s^t\Vert \eta'(\tau)\Vert_\infty\, d\tau,
$$
so for any partition $a=t_0<t_1<\cdots<t_n=b$ of $[a,b]$, we have
$$
\sum_{k=0}^{n-1} \Vert\eta(t_{k+1})-\eta(t_k)\Vert_\infty \leq
\sum_{k=0}^{n-1}\int_{t_k}^{t_{k+1}} \Vert \eta'(\tau)\Vert_\infty\, d\tau =
\int_a^b \Vert\eta'(\tau)\Vert_\infty\, d\tau.
$$
Since $\ell(\eta)$ equals the supremum of the sums on the left hand side over all partitions of $[a,b]$, the lemma follows.
\end{proof} 

Note that, on the set $\gamma^{-1}(K_j)$, the curve $\Phi\circ\Gamma$ coincides with $g\circ\gamma$. 
Thus for almost every $t \in \gamma^{-1}(K_j)$ we have
$$
(\Phi\circ\Gamma)'(t) = (g\circ\gamma)'(t) = 0.
$$
This is an easy consequence of \eqref{eq122}
since $\gamma$ is a curve in $\{\rho\}\times[0, d]^{k-j}$. 
Hence the length of the curve $\Phi\circ\Gamma$ satisfies
\begin{eqnarray*}
\ell(\Phi\circ\Gamma) 
& \leq &
\int_0^{\ell(\gamma)}\Vert (\Phi\circ\Gamma)'(t)\Vert_\infty \, dt \\
& \leq &
(MC_q^2L) \H^1\left([0,\ell(\gamma)]\setminus \gamma^{-1}(K_j) \right) \\
& \leq &
C(k)MC_q^2Lm^{-1}d
\end{eqnarray*}
which proves \eqref{eq367}. The proof is complete.
\end{proof}

\end{document}